\documentclass[12pt]{amsart}
\usepackage{amssymb}
\usepackage{mathptmx}
\usepackage{hyperref}
\usepackage[latin1]{inputenc}

\usepackage[all]{xy}

\newtheorem{step}{Step}
\DeclareMathOperator{\Span}{Span}
\DeclareMathOperator{\Supp}{Supp}
\DeclareMathOperator{\Ker}{Ker}
\DeclareMathOperator{\Hom}{Hom}

\newtheorem{thm}{\bf Theorem}[section]
\newtheorem{lem}[thm]{\bf Lemma}

\theoremstyle{definition}

\newtheorem{rem}[thm]{\bf Remark}
\newtheorem{exa}[thm]{\bf Example}

\title{On the weak Lefschetz Property of graded modules over $K[x,y]$}
\keywords{Lefschetz properties, monomial ideals, indecomposable module}

\author{Giuseppe Favacchio}
\address{%
Dipartimento di Matematica\\
Universit$\grave{a}$ degli studi di Catania\\}
\email{favacchio@dmi.unict.it}
\author{Phong Dinh Thieu}
\address{%
Institut f\"ur Mathematik\\
Universit\"at Osnabr\"uck\\}
\email{pthieudi@uni-osnabrueck.de}

\date{2012/03/25}

\begin{document}

\maketitle

% Abstract
\begin{abstract}
It is known that graded cyclic modules over $S=K[x,y]$ have the Weak Lefschetz Property (WLP). This is not true for non-cyclic modules over $S$. The purpose of this note is to study which conditions on $S$-modules ensure the WLP. We give an algorithm to test the WLP for graded modules with fixed Hilbert function. In particular, we prove that indecomposable graded modules over $S$ with the Hilbert function $(h_0,h_1)$ have the WLP.
\end{abstract}

% The article itself
\section{Introduction}
\hspace{18pt} Let $S$ be the standard graded polynomial ring over a field $K$ of characteristic zero. Let $M$ be a standard graded module over $S$. The module $M$ is said to have the \textit{Weak Lefschetz Property} (WLP for short) if there exists a linear form  $\ell\in S_1$, called \textit{Lefschetz element}, such that for each degree $i$, the multiplication map $ \times\ell : M_i \rightarrow M_{i+1}$ has maximal rank, i.e., the map is either injective or surjective.

The Weak Lefschetz Property has been studied extensively for especially the relation to the Hilbert function (see, e.g., \cite{HMNW}, \cite{MZ}, and \cite{ZZ} for more details).  Up to now, most of the known results about the WLP concern standard graded Artinian $K$-algebra over $S$. It is not known much about the WLP for standard graded modules over $S$, so the case of low dimension is still interesting.

In this paper, we study the WLP for standard graded modules over the standard graded polynomial ring $S=K[x,y]$, where $K$ is an infinite field. We are interested in conditions ensuring the WLP for a graded module over $S$. It is known that cyclic $S$-modules have the WLP and this is not true for non-cyclic $S$-modules, as we see in the following example:
\begin{exa}
Let $M=S/I_1\oplus S/I_2$ be a standard graded module over $S$, where $I_1=(x^2, xy, y^2)$ and $I_2=(x, y)$. The Hilbert function $HF_M$ of $M$ is given by $HF_M(0)=2$, $ HF_M(1)=2$ and zero otherwise. The multiplication by any generic linear form from $M_0$ to $M_1$ can not be injective because it is not injective on the second component, nor surjective because it is not surjective on the first component.
\end{exa}
One more example to see that there exists a module with the same Hilbert function as above and it has the WLP:
\begin{exa}
Let $M=S/I_1\oplus S/I_2$, where $I_1 = I_2=(x, y^2)$. The Hilbert function of $M$ is $HF_M(0)=2$, $ HF_M(1)=2$ and zero otherwise. The multiplication by $y$ is injective (and surjective).
\end{exa}

In Section \ref{sec dete} we study the WLP of graded modules in the case the Hilbert functions are nonzero only in two consecutive degrees. We present conditions of concrete matrices to ensure the WLP. Moreover, we give an algorithm in Section \ref{sec algo} to test the WLP for fixed graded modules. As an application, we prove in Section \ref{sec inde} that indecomposable graded modules over $S$ with Hilbert function $(h_0,h_1)$ have the WLP. We also find out an equivalent condition to ensure the WLP for indecomposable graded modules in the general situation and construct an example in which an indecomposable graded module with a non-decreasing Hilbert function does not have the WLP.
%%%%%%%%%%%%%%%%%%%%%%%%%%%%%%%%%%%%%%%%%%%%%%%%%%%%%%%%%%%%%%%%%%%%%%%%%%%%55
\section{Determinant conditions to ensure the WLP}
\label{sec dete}
\paragraph{}
Let $S=K[x,y]$ be the standard graded polynomial ring over an infinite field $K$. Let $M$ be a standard graded module over $S$. We study in this section the WLP of $M$ in the case the Hilbert function of $M$ is  $HF_M=(h_0, h_1)$, where $1\leq h_0\leq h_1$.
\begin{rem}
If the Hilbert function of $M$ is  $HF_M=(h_0, h_1)$, where $h_0=h_1=n\geq 1$ and $M$ has a minimal generator of degree 1 then $M$ does not have the WLP. In fact, the vector space generated by $\ell M_0$, where  $\ell$ is a general linear form, has dimension strictly less than $n$. Therefore, the multiplication map by a general linear form can not be injective or surjective.
\end{rem}
As noted above, we only need to study the case where $M$ is minimally generated by elements of degree $0$.

\begin{lem}\label{l1}
Let $M$ be a finitely generated standard graded module with a minimal system of generators $e_1, \ldots, e_n$ of degree $0$ and the Hilbert function $HF_M=(h_0,h_1)$, where $n=h_0\leq h_1$. If $M$ has the WLP then there exists a linearly independent set in $M_1$ of the form $\{z_1e_1, \ldots, z_ne_n\},$ where $z_i\in \{x,y\}$ for $1\leq i\leq n$.
\begin{proof}
Since $M$ has the WLP and $h_0\leq h_1$, the multiplication map by a Lefschetz element is injective. This is also true for every submodule of $M$. We prove the statement by induction on $n$.

For the case $n=1$, since $M$ has the WLP, one of $xe_1$ and $ye_1$ must be non-zero and the statement holds obviously.

Assume that the statement holds for $n=1,\ldots, k$. We turn to prove that it is true for the case $n=k+1$. Observe that $N= (e_1, \ldots, e_k)$ is a submodule of $M$ and $HF_N=(k,k')$. Since $M$ has the WLP, $N$ has the WLP and $k\leq k'$. By the induction hypothesis, we can choose a linearly independent set of the form
$A=\{ z_1e_1, \ldots, z_ke_k \}$ in $N_1$ with $z_i\in \{x,y\}$ for $1\leq i\leq k$. Let $V=\Span_K A$. For $v\in V$, $v=\beta_1z_1e_1+\ldots+\beta_k z_ke_k$, we denote the set $\Supp(v)=\{j: \beta_j\neq 0\}$. We aim to show how to build a linearly independent set of $n=k+1$ elements.

\textit{Case 1.} If one of elements $xe_{k+1}$ and $ye_{k+1}$ is not in $V$, we add that element to $A$ and we get a linearly independent set satisfying the conditions of the statement.

\textit{Case 2.} Assume that both $xe_{k+1}, ye_{k+1}$ are in $V$. Then one of them must be non-zero, otherwise we get $S_1 e_{k+1}=0$ and then all multiplication maps by linear forms can not be injective. Therefore, $\Supp(xe_{k+1})\cup \Supp(ye_{k+1})\neq \emptyset$. Assume that every set of the form $\{z'_1e_1,\ldots,z'_{k+1}e_{k+1}\}$, where $z'_i\in \{x,y\}$ for $i=1,\ldots,k+1$, is linearly dependent.

If $\Supp(xe_{k+1})\cup \Supp(ye_{k+1})= \{1,\ldots,k\}$, then for $1\leq i\leq k$, the set
$$
B=(A\setminus \{z_ie_i\})\cup \{t_ie_i\} \cup \{z_{k+1}e_{k+1}\}
$$
is linearly dependent, where $t_i\in \{x,y\} \setminus\{z_i\}$, $z_{k+1}=x$ if $i\in \Supp(xe_{k+1})$ and $z_{k+1}=y$ else. Moreover, the set $B'=(A\setminus \{z_ie_i\}) \cup \{z_{k+1}e_{k+1}\}$ is still linearly independent because $|B'|=k$ and $\Span_K B'= V$. Therefore, $t_ie_i\in V$ for all $i=1,\ldots,k$. This implies that $M_1=V$ and $\dim_K M_1=\dim_K V=n-1<n$. Hence $M$ does not have the WLP, a contradiction.

If $\Supp(xe_{k+1})\cup \Supp(ye_{k+1})\neq \{1,\ldots,k\}$, without loss of generality, we can assume that $\Supp(xe_{k+1})\cup \Supp(ye_{k+1})= \{1,\ldots,s\}$ where $s< k$. We aim to prove $t_ie_i\in V$ for all $i=1,\ldots,k$. By the same proof as above, we get $t_ie_i\in V$ for $i=1,\ldots,s$.

Let $C= \bigcup_{i=1}^s \Supp(t_ie_i)$. For $j>s$, if $j\in C$, say $j\in \Supp(t_1e_1)$, then the set
$$
D=(A\setminus\{z_1e_1,z_je_j\})\cup \{t_1e_1,t_je_j,z_{k+1}e_{k+1}\}
$$
is linearly dependent and the set
$$
D'=(A\setminus\{z_1e_1,z_je_j\})\cup \{t_1e_1,z_{k+1}e_{k+1}\}
$$
is linearly independent.
Moreover, $\Span_K D'=\Span_K A=V$. Therefore, $t_je_j\in V$. We repeat the above process for the set $C'=C\cup \bigcup_{j\in C}\Supp(t_je_j)$. Note that for $p\in C'\setminus C$, say $p\in \Supp(t_je_j)$ where $j>s$ and $j\in \Supp(t_1e_1)$, then the set
$$
E=(A\setminus\{z_1e_1,z_je_j,z_pe_p\})\cup \{t_1e_1,t_je_j,t_pe_p,z_{k+1}e_{k+1}\}
$$
is linearly dependent, while the set
$$
E'=(A\setminus\{z_1e_1,z_je_j,z_pe_p\})\cup \{t_1e_1,t_je_j,z_{k+1}e_{k+1}\}
$$
is linearly independent. Moreover, $\Span_K E'=\Span_K A=V$. Hence $t_pe_p\in V$. This process will stop after a finite number of steps. Let $C_0$ be the final union set of indices. Then $C_0= \{1,\ldots,k\}$. Otherwise, we get
$$
t_je_j\in \Span_K\{z_ie_i: i\in C_0\} \text{ for } j\in C_0
$$
and the submodule $N=\sum_{i\in C_0}Se_i+Se_{k+1}$ has dimension
$$
\dim_K N_1=\#C_0< \dim_K N_0=\#C_0+1.
$$
Therefore, $N$ does not have the WLP, so does $M$. Hence we have $t_ie_i\in V$ for all $i=1,\ldots,k$. This implies that $M_1=V$ and $\dim_K M_1=\dim_K V=n-1<n$. Hence $M$ does not have the WLP, a contradiction. This concludes the proof.
\end{proof}
\end{lem}

In the following, we aim to give a procedure to verify if $M$ has the WLP. By Lemma \ref{l1}, we can assume that
$
\{xe_1, \ldots, xe_r, ye_{r+1}, \ldots, ye_n\}
$
is a basis of $M_1$. The multiplication maps by the variables:
$$
\times x : M_0 \rightarrow M_1,\;\;\;\;\; \times y: M_0 \rightarrow M_1
$$
are morphisms between vector spaces of the same dimension. Let $A$, $B$ be their matrices, respectively. Then we have
\begin{displaymath}
A = \left(\begin{array}{c|c}
I_r  &  A'  \\
\hline
0  &  A''
\end{array}\right),\;\;\;
B = \left(\begin{array}{c|c}
B'  &  0  \\
\hline
B''  &  I_{n-r}
\end{array}\right)
\end{displaymath}
where $I_{r}$ and $I_{n-r}$ are the identity matrices of the sizes $r$ and $n-r$, respectively, and $0$ is the null matrix having the appropriate size.

It is clear that $M$ has the WLP if and only if there exist $\alpha, \beta \in K $ such that
$$
|\alpha A + \beta B|\neq 0.
$$

Note that if $|A| \neq 0$ we can choose $\alpha=1 , \beta=0  $, similarly if $|B| \neq 0$,  in these cases $M$ has the WLP.
Thereafter we can assume $|A| = |B| = 0 $, so we only need to check the existence of $\alpha\neq 0$ and $\beta \neq 0$ such that  $|\alpha A + \beta B|\neq 0.$ We have:
\begin{eqnarray*}
|\alpha A + \beta B| &=& \left|\left(\begin{array}{c|c}

\alpha I_r +\beta B' &  \alpha A'  \\

\hline

\beta B''  &  \alpha A'' +\beta I_{n-r}
\end{array}\right)\right| \\
&=&  \left|\left(\begin{array}{c|c}

\frac{\alpha}{\beta} I_r + B' &   A'  \\

\hline

 B''  &   A'' +\frac{\beta}{\alpha} I_{n-r}
\end{array}\right)\right| \alpha ^{n-r} \beta^r.
\end{eqnarray*}

Let $\gamma =\frac{\alpha}{\beta}$,  the determinant $|\alpha A + \beta B|$ is a polynomial of the form $\frac{1}{\gamma^d}p(\gamma)$ in $K[\gamma,\frac{1}{\gamma}]$, where $p(\gamma)\in K[\gamma]$. If  $p(\gamma)$ is the zero polynomial then $M$ does not have the WLP, otherwise there always exists $\tau \in K$ such that $p(\tau)\neq 0$. In this case $M$ has the WLP with a Lefschetz element $\ell = \tau x + y$.

\begin{exa}
Let $M=(( x^6, x^2y^4, x^3y^3) +I)/I\subseteq S/I$ be a graded $S$-module, where $I=(x,y)^8+(x^2y^5, x^4y^3)$ and the degrees are shifted to $6$.

Observe that $xm_1$ and $ym_1$ are linearly independent and not in the space
$$
\Span_K\{xm_2, xm_3, ym_2, ym_3 \}.
$$
By changing the basis of $M$, we have
\begin{eqnarray*}
M_0 &=&\Span_K\{(x^6+x^2y^4)+I, (x^6-x^2y^4)+I, (x^3y^3)+I\},\\
M_1 &=& \Span_K\{(x^7+x^3y^4)+I, x^6y+I, x^3y^4+I\}.
\end{eqnarray*}
Set $e_1= x^6+I, e_2=x^2y^4+I, e_3=x^3y^3+I$. We get that $\{xe_1, ye_2, ye_3\}$ is a basis of $M_1$ which is of the form as in Lemma \ref{l1} and $ye_1= xe_2$, $xe_2= xe_1-2ye_3 $ and $ xe_3= 0$. The matrices given by the maps $\times x$ and $\times y$ are:
$$
A = \left(\begin{array}{ccc}
1&1 & 0  \\
0  &0 & 0\\
0&-2&0
\end{array}\right), \;\;\;\; B = \left(\begin{array}{ccc}
0  &0 & 0\\
1&1 & 0  \\
0&0&1
\end{array}\right).
$$
By computing $\alpha A + \beta B$, and setting $\tau = \frac{\beta}{\alpha}$ we obtain the matrix:
$$
\left(\begin{array}{ccc}
\tau & 1 & 0  \\
1  & \frac{1}{\tau} & 0\\
0&-2&   \frac{1}{\tau}
\end{array}\right)
$$
which has determinant equal to zero for all $\tau$, so $M$ does not have the WLP.

From the above note, we can construct a procedure to ensure the WLP of a given graded module $M$ with the Hilbert function $(n,n)$ over $K[x,y]$ as in the following:

\emph{Step 1:} Take an arbitrary minimal system of generators $e_1,\ldots,e_m$ of $M$ in degree $0$. By Lemma \ref{l1}, we check the linearly independent property of all sets of the form $\{z_1e_1, \ldots, z_ne_m\},$ where $z_i\in \{x,y\}$ for $1\leq i\leq m$.
If all sets are linearly dependent, then we conclude that $M$ does not have the WLP. Else, we turn to Step 2

\emph{Step 2:} After changing the indices, we can choose a basis of $M_1$ of the form
$$
xe_1, \ldots, xe_r, ye_{r+1}, \ldots, ye_m.
$$

Compute the matrices $A$, $B$ of the multiplications by $x$, $y$.

\emph{Step 3:} Construct a matrix $C$ from $A$ and $B$ by taking the columns $1^{st},\ldots,r^{th}$ of $B$ to be the $r$ first columns of $C$ and taking the columns $(r+1)^{th},\ldots m^{th}$ of $A$ to be the remain columns of $C$. Let $D$ be the matrix inducing from $C$ by replacing the diagonal by $\tau$ for the first $r$ elements in the upper of diagonal and $1/\tau$ for the $(m-r)$ remain elements in the lower of the diagonal. Let $P(\tau)\in K[\tau]$ be the determinant of $D$.

If $P(t)$ is zero polynomial, then $M$ does not have the WLP. Otherwise $M$ has the WLP.
\end{exa}
%%%%%%%%%%%%%%%%%%%%%%%%%%%%%%%%%%%%%%%%%%%%%%%%%%%%%%%%%
\section{Algorithm to check the WLP}
\label{sec algo}
\paragraph{}
\hspace{4pt}  Let $S=K[x,y]$ be the standard graded polynomial ring over a field $K$ of characteristic zero. In this section we develop an algorithm to check the WLP for Artinian $S$-modules.

Let $ M=M_0\oplus M_1 \oplus \cdots \oplus M_s $ be an Artinian graded $S$-module with Hilbert function $HF_M=(h_0, \ldots, h_s)$. The module $M$ has the WLP if for each degree $i$, there exists a linear form $\ell_i$ such that the map $\times \ell_i  : M_i \rightarrow M_{i+1}$ has maximal rank.  More precisely, the set of all Lefschetz elements is a non-empty Zariski open set for each degree and their intersection is non-empty, so we can choose an element in this intersection to be a Lefschetz element for $M$.

Fixed a degree $i$, if $h_i\leq h_{i+1}$ then the WLP of $M$ implies that there exists a linear form $\ell_i$ such that the map $\times\ell_i : M_i \rightarrow M_{i+1} $ is injective. Otherwise if $h_i > h_{i+1}$, there exists a linear form $\ell_i$ such that the multiplication by $\ell_i$ in degree $i$ is surjective, but this is equivalent to the fact that the multiplication map by $\ell_i$ of the dual space is injective. Thus if $h_i > h_{i+1}$ we can dualize and search for an injective multiplication for the dual space. Therefore, it is enough to study modules of the type $M=M_0\oplus M_1$ with Hilbert function $HF_M=(h_0, h_1)$ and $h_0\leq h_1$.

At first, we consider the multiplication by $x$ and $y$:
$$
\times x : M_0 \rightarrow M_1,\;\;\;\;\; \times y: M_0 \rightarrow M_1.
$$
\begin{step}
If $\Ker(\times x)= (0) $, then it is clear that $M$ has the WLP and $x$ is a Lefschetz element.
\end{step}
\begin{step}Similarly, we consider $\times y$ and if $\Ker(\times y)= (0) $ then  $M$ has the WLP.
\end{step}
So, we can assume that $\dim_K \Ker(\times x) = r > 0 \;\;\text{and}\;\; \dim_K \Ker(\times y) = s > 0.$
\begin{step} 
If $\Ker(\times x) \cap \Ker(\times y) \neq (0)$ then  $M$ does not have the WLP.
\end{step}
\begin{proof}
In fact, if  $m \in \Ker(\times x) \cap \Ker(\times y)$ we have $\ell m = 0$ for all $\ell \in S_1$. Then the multiplication by any linear form $\ell\in S_1$ can not be injective
\end{proof}

Now we can assume that $\Ker(\times x) \cap \Ker(\times y) = (0)$. By considering the subspace $y \Ker(\times x)+ x \Ker(\times y)\subseteq M_1$, we continue with the following steps.
\begin{step} 
\label{step4}
If  $y \Ker(\times x)\cap x \Ker(\times y) \neq (0)$, then 
$$
\dim_K (y \Ker(\times x)+ x \Ker(\times y)) < r+s. 
$$ 
In particular, $M$ does not have the WLP .
\end{step}
\begin{proof}
It is clear that the dimension of image of a vector space is always less than the dimension of the vector space, so we have
$$
\dim_K (y\Ker(\times x))\leq \dim_K \Ker(\times x)=r
$$ 
and  
$$
\dim_K (x\Ker(\times y))\leq \dim_K \Ker(\times y)=s.
$$
Now since $y \Ker(\times x)\cap x \Ker(\times y) \neq (0)$, we get
$$
\dim_K (y \Ker(\times x)+ x \Ker(\times y))<\dim_K (y \Ker(\times x))+ \dim_K (x\Ker(\times y))\leq r+s.
$$
Next let $N= \Ker(\times x)+\Ker(\times y)\subset M_0$. For a linear form $\ell\in S_1$, we have $\ell N\subseteq y \Ker(\times x)+ x \Ker(\times y)$. Moreover,
$$
\dim_K N=r+s\text{ and }\dim_K (y \Ker(\times x)+ x \Ker(\times y))< r+s.
$$
Hence the multiplication map by $\ell$ is not injective. This concludes the proof.
\end{proof}
%--------------------------------------------------------------------------
Next we claim that we only need to consider a Lefschetz element of the form $\alpha x + \beta y $ with $\alpha$ and $ \beta$ are different from zero. The existence of such Lefschetz element can be seen simply by the following:
\begin{lem}
Let $M$ be a graded $S$-module with  $HF_M =( h_0, h_1 )$, $ h_0\leq h_1 $, and $M$ has the WLP. Then there exist $\alpha, \beta \neq 0$ such that $\ell = \alpha x + \beta y $ is a Lefschetz element of $M$.
\end{lem}
\begin{proof} Suppose by the contrary that for every $\alpha, \beta \neq 0 $ the multiplication by $\alpha x + \beta y $ is not injective. Without loss of generality, we assume that $ \ell = y$ is a Lefschetz element.
Let $L_r = x+ r y$, $r\in \mathbb{N}\subset K$. For each $r \in \mathbb{N} $ take $0 \neq m_r \in Ker (\times L_r)$ and fix $t \in \mathbb{N}$, we aim to prove that $\{ m_r \}_{r \leq t }$ is an independent set by induction on $t$.

If $t=2$, let $\lambda_1, \lambda_2\in K$ such that $\lambda_1 m_1 + \lambda_2 m_2 = 0$, then we get
\begin{eqnarray*}
L_1(\lambda_1 m_1 + \lambda_2 m_2)&=&\lambda_2(x+y)m_2= L_2 \lambda_2m_2- \lambda_1 y m_2= - \lambda_1 y m_2=0;\\
L_2(\lambda_1 m_1 + \lambda_2 m_2)&=&\lambda_1(x+2y)m_1= L_1 \lambda_1m_1+ \lambda_2 y m_1=\lambda_2 y m_1=0.
\end{eqnarray*}
Since the multiplication map by $y$ is injective, we have $- \lambda_2  m_2 = \lambda_1 m_1=0$ and then $\lambda_2  = \lambda_1 =0$.

Let $\{\lambda_1,\ldots, \lambda_{t+1} \}$ such that $ \sum_{r=1}^{t+1} \lambda_r m_r =0 $. We have
\begin{eqnarray*}
L_{t+1} \sum_{r=1}^{t+1} \lambda_r m_r&=& (x+(t+1)y) \sum_{r=1}^{t} \lambda_r m_r \\
&=&\sum_{r=1}^{t}L_r\lambda_r m_r + \sum_{r=1}^{t} \lambda_r(t+1-r)y m_r\\
&=& y \sum_{r=1}^{t} \lambda_r(t+1-r) m_r =0.
\end{eqnarray*}
Since the map $(\times y)$ is injective, we get $\sum_{r=0}^{t+1} \lambda_r(t-r) m_r =0 $. By the induction hypothesis, we have $\lambda_r = 0$ for all $r = 1, \ldots, t$ and then $\lambda_{t+1}=0$. This implies that the dimension of $M_0$ is infinite and so we can conclude the proof.
\end{proof}
Now we assume that $\dim_K (y \Ker(\times x)+ x \Ker(\times y)) = r+s $ and denote $\overline{M}=M/N$, where $N$ is the graded submodule generated by $\Ker(\times x)+ \Ker(\times y)$. Then
$$
\overline{M} = \overline{M_0}\oplus\overline{M_1} = M_0 / (\Ker(\times x)+  \Ker(\times y)) \oplus M_1/(y \Ker(\times x)+ x \Ker(\times y))
$$
and $HF_{ \overline{M} }= (h_0-r-s,  h_1-r-s)$, which is still not decreasing.
\begin{step}
\label{step5}
$M$ has the WLP if and only if $\overline{M}$ has the WLP.
\end{step}
\begin{proof}
($\Rightarrow$): Let $\ell = \alpha x + \beta y$, $\alpha, \beta \neq 0$, be a Lefschetz element of $M$ and $\overline{m}\in \overline{M_0}$ such that $\ell \overline{m} = \overline{0}$. There exist $m_0$ in $\Ker(\times x)$ and  $m_1$ in $\Ker(\times y)$ such that
$$
\ell m = y m_0 + x m_1 = \ell \frac{1}{\beta} m_0 + \ell \frac{1}{\alpha} m_1 = \ell( \frac{1}{\beta} m_0 +  \frac{1}{\alpha} m_1).
$$
Since the multiplication by $\ell$ is injective, $m=  \frac{1}{\beta} m_0 +  \frac{1}{\alpha} m_1$. Hence $\overline{m}=\overline{0}$.

($\Leftarrow$): Let $\ell = \alpha x + \beta y $ with $\alpha, \beta \neq 0$ be a Lefschetz element of $\overline{M}$ and $m \in M_0$ such that $\ell m = 0$. Then $\ell \overline{m} = \overline{0}$. Since the map $\times \ell$ is injective, we get $ \overline{m} = \overline{0}$, i.e, $m \in \Ker(\times x)+  \Ker(\times y)$, say $ m =  m_0 +  m_1$, where $m_0\in\Ker(\times x)$,  $m_1\in \Ker(\times y)$. We have $\ell m = \beta  y m_0 +\alpha x m_1 = 0 $. Therefore, $\beta  y m_0 = -\alpha x m_1$. Moreover, $y \Ker(\times x)\cap x \Ker(\times y) = (0)$, so $m\in \Ker(\times x) \cap \Ker(\times y) = (0)$.
\end{proof}
%%%%%%%%%%%%%%%%%%%%%%%%%%%%%%%%5
Combining all steps together, we have the following algorithm to check the WLP:
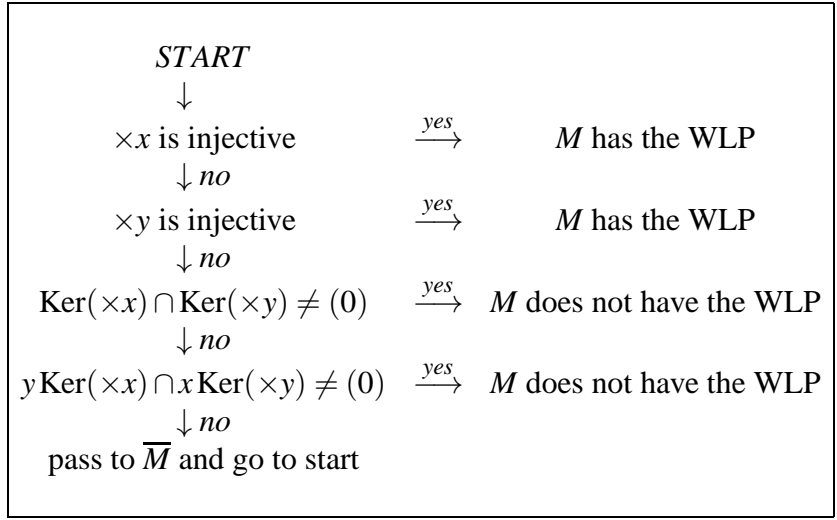
\begin{center}
\begin{tabular}{|ccc|}
\hline
&&\\
$START$  &  & \\
$\downarrow\phantom{no} $  &  & \\
$\times x$ is injective &$\stackrel{yes}{\longrightarrow}$  &  $M$ has the WLP   \\
$\downarrow no$  &  & \\
$\times y$ is injective &$\stackrel{yes}{\longrightarrow}$  &  $M$ has the WLP   \\
$\downarrow no$  &  & \\
$\Ker(\times x)\cap \Ker(\times y) \neq (0)$  &$\stackrel{yes}{\longrightarrow}$  &  $M$ does not have the WLP   \\
$\downarrow no$  &  & \\

$ y\Ker(\times x)\cap x\Ker(\times y) \neq (0)$
	 &$\stackrel{yes}{\longrightarrow}$  &   $M$ does not have the WLP   \\
$\downarrow no$  &  & \\
 pass to $\overline{M}$ and go to start  &  & \\
&&\\
\hline
\end{tabular}
\end{center}

Note that if the first four steps give us negative answers then we can replace $M$ by $\overline{M}$ and back to the start. Moreover, this algorithm ends after a finite number of steps because after each cycle the Hilbert function of the module decreases by at least two in each degree.
\begin{exa}
Let $M= (( x^8, x^6y^2, x^4y^4, x^2y^6, y^8) + I)/I\subseteq S/I$, where $I$ is the ideal defined by $I=(x,y)^{10} +(x^8y, x^7y^2, x^9- x^2y^7, x^6y^3-x^5y^4 )$. After shifting the degree of $M$, it has the Hilbert function $HF_M=(5,6)$. We denote the $i$-th generator of $M$ by $m_i$. Then we have:
\begin{enumerate}
	\item  $\Ker(\times x) = \langle m_2\rangle \neq 0$,
	\item  $\Ker(\times y ) = \langle m_1\rangle \neq 0 $,
	\item  $\Ker(\times y )\cap \Ker(\times y )= 0$,
	\item  $y\Ker(\times x)\cap x\Ker(\times y) = 0 $.
\end{enumerate}
Passing to $\overline{M}$, we have $\overline{M} = ((x^4y^4, x^2y^6, y^8) + I')/I'$, where $I'=I+(x^9, x^6y^3)$, and $HF_{\overline{M}}=(3,4)$. Repeat the process again
\begin{enumerate}
	\item  $\Ker(\times x) = \langle m_3\rangle \neq 0$,
	\item  $\Ker(\times y ) = \langle m_4\rangle \neq 0 $,
	\item  $\Ker(\times y )\cap Ker(\times y )= 0$,
	\item  $ y\Ker(\times x)\cap x\Ker(\times y) = 0 $.
\end{enumerate}
Finally, passing to $\overline{\overline{M}}$ we have
$\overline{\overline{M}}= (( y^8) +I'')/I''$, where $I''=I'+(x^4y^5, x^3y^6)$, $HF_{\overline{\overline{M}}}=(1,2)$ and $(\times x)$ is injective. Therefore, $M$ has the WLP.
\end{exa}
%%%%%%%%%%%%%%%%%%%%%%%%%%%%%%%%%%%%%%%%%%%%%%%%%%%%%%%%%%%%%
\section{Indecomposable modules and the WLP}
\label{sec inde}
\paragraph{}

In this section, we study the WLP of indecomposable modules over the standard graded polynomial ring $S=K[x,y]$, where $K$ is a field of characteristic zero. Note that indecomposable modules play an important role in the study of WLP in general situation, as described in the following observation.

\begin{rem}
Let $M$ be a graded $S$-module. Suppose that $M$ can be decomposed as a direct sum of indecomposable submodules $M= N_1\oplus N_2\oplus\ldots\oplus N_t$. Then $M$ has the WLP if and only if all the direct summands $N_i$ have the WLP and their Hilbert functions have the same behavior. More precisely, for each degree, if the Hilbert function of one summand is strictly increasing (strictly decreasing) then the Hilbert functions of the other summands are also strictly increasing (strictly decreasing).
\end{rem}

Next by using the algorithm in Section \ref{sec algo}, we have the following result:
\begin{thm}\label{thmind}
Let $M$ be an Artinian graded $S$-module such that every its submodule has a non-decreasing Hilbert function, then $M$ has the WLP.
\end{thm}
\begin{proof}
Let $HF_M=(h_0, \ldots, h_d)$. To ensure the WLP of $M$ we must ensure that there is an injective map $\times\ell : M_i\rightarrow M_{i+1}$ for $i=0, \ldots, d$, where $\ell \in S_1$.

Fix $i$, let $M'=\bigoplus_{j>i+1} M_j$ and $M_{\geq i}=\bigoplus_{j\geq i} M_j$. Then $M_{\geq i}=M_i\oplus M_{i+1}\oplus M'$ and $M'$ is a submodule of $M_{\geq i}$. Let
$N = (M_{\geq i}/M')(-i) \subseteq (M/M')(-i).$ 
The Hilbert function of $N$ is $HF_N=( h_i, h_{i+1} )$. Observe that the map $\times\ell: M_i\longrightarrow M_{i+1}$ is injective if and only if the map $\ell|_{N_0}: N_0 \longrightarrow N_1$ is too. So we only need to prove that $N$ has the WLP.

By the hypothesis on submodules of $M$, we get that every submodule of $N$ has a non-decreasing Hilbert function. Hence to check the WLP for $N$ we can use directly the algorithm in Section \ref{sec algo}.

Suppose that the first two steps in the algorithm give us negative answers. Then by Step \ref{step4} and the fact that every submodule of $N$ has a non-decreasing Hilbert function, we have 
$$
\Ker(\times x) \cap \Ker(\times y) = (0)\text{ and } y\Ker(\times x) \cap x\Ker(\times y) = (0).
$$
So by Step \ref{step5} in the algorithm, we aim to prove the WLP for $\overline{N}=N/T$ where $T=\langle  \Ker(\times x) + \Ker(\times y)\rangle$ is a submodule of $N$. 

Note that by using the algorithm repeatedly, we only need to confirm that every submodule of $\overline{N}$ has a non-decreasing Hilbert function. Let $\overline{P}$ be a submodule of $\overline{N}$. Then $P + T$  is a submodule of $N$ and it has a non-decreasing Hilbert function, we prove that $\overline{P}$ also has a non-decreasing Hilbert function. In fact, let $Q= P \cap T$, this is a submodule of $N$, so it has a non-decreasing Hilbert function. The proof follows from:
\begin{eqnarray*}
 \dim_K  \overline{P_0}\,&=& \dim_K P_0 - \dim_K ( \Ker(\times x) + \Ker(\times y)) +\dim_K Q_0\\
&=&\dim_K P_0  + \dim_K Q_0 -(r+s),\\
 \dim_K  \overline{P_1}&=& \dim_K P_1 - \dim_K ( y\Ker(\times x) + x\Ker(\times y)) +\dim_K Q_1\\
&=&\dim_K P_1  + \dim_K Q_1 -(r+s),
\end{eqnarray*}
where $r=\dim_K ( \Ker(\times x)$ and $s=\dim_K ( \Ker(\times y)$.
\end{proof}
\begin{rem} The converse of Theorem \ref{thmind} is also true if $M$ has a non-decreas\-ing Hilbert function. In fact, let $\ell$ be a Lefschetz element of $M$, we have that the multiplication by this linear form is injective. Hence if there exists a submodule $N$ of $M$ such that $HF_N$ decreases in two consecutive degrees, then the multiplication map $\times\ell_{|_{N}}$ is not injective. So we get that every submodule of $M$ has a non-decreasing Hilbert function.
 \end{rem}

Now we aim to show that if $M$ is a graded  indecomposable S-module with Hilbert function $HF_M=(h_0,h_1)$, then $M$ has the WLP.
\begin{lem}\label{lemind2}
Let $M$ be a graded indecomposable S-module with a non-decreas\-ing Hilbert function, $HF_M=(h_0,h_1)$. Then every submodule of $M$ has a non-decreasing Hilbert function.
\end{lem}
\begin{proof}
Let $N$ be a submodule of $M$ with $HF_N=(r,s)$, we can assume that $N $ has minimal generators only in degree zero. We prove the statement by induction on $n-r$ where $n=h_0$.

For the case $r=n$, the statement holds obviously, because all minimal generators of $M$ are in degree zero, then this implies $M=N$. Assume that the statement is true for each submodule minimally generated by $r > t$ elements, we claim that it true for the case $r=t$.

If $r > s $, by the induction hypothesis the statement is true for the submodule $N + \langle e \rangle $, for each $e\in M_0 \backslash N_0$. This means that $xe$ and $ye$ are linearly independent and the submodule generated by $e$ does not intersect with $N$, i.e. $\langle e\rangle \cap N = (0)$. Hence $s = r-1$.

Next we claim that $M=N\oplus \langle M_0\backslash N_0 \rangle$, which contradicts to the hypothesis on $M$. In fact, if $m \in N\cap \langle M_0 \backslash N_0 \rangle$, then $m \in M_1$, so $m$ is not a minimal generator. Since $m \in N$, there is $e_A\in N_0$ such that $\ell_A e_A = m$. Similarly, since $m\in \langle M_0\backslash N_0 \rangle$, there is $e_B\in  M_0\backslash N_0$ such that $\ell_B e_B = m$, for some $\ell_A, \ell_B \in S_1$.

By using the same argument for the submodules $N + \langle e_B \rangle$, we get a contradiction. This concludes the proof.
\end{proof}

Now we are able to prove the main result of this section:
\begin{thm}\label{thmind2}
Let $M$ be a graded indecomposable S-module with Hilbert function $HF_M=(h_0,h_1)$. Then $M$ has the WLP.
\end{thm}
\begin{proof}
If $h_0\leq h_1$, it is followed from Lemma \ref{lemind2} and Theorem \ref{thmind}.

For the case $h_0>h_1$, the dual module $\Hom_K(M,K)$ of $M$ will be an indecomposable module with a non-decreasing Hilbert function, see \cite{K}. Therefore, $\Hom_K(M,K)$ has the WLP, hence $M$ has WLP.
\end{proof}
Theorem \ref{thmind2} is false if we consider an indecomposable module with a long enough Hilbert function. This can be seen by the following example in which an indecomposable module with a Hilbert function of length 4 does not have the WLP.

\begin{exa}
Let $M= (( y, x^4 ) + I)/I\subset S/I$ where the degree are shifted to 1 and $I = (y^3, x^2y^2)+ (x,y)^6.$
The Hilbert function of $M$ is
$$
HF_M=(1,2,2,2,2).
$$
Then $M$ does not have the WLP. In fact, $M$ has a minimal generator of degree 4, so the multiplication map by any linear form from $M_3$ to $M_4$ can not be surjective because the minimal generator $x^4+I$ is not an image of any element in $M_3$. Since the Hilbert function $HF_M(3)=HF_M(4)=2$, this multiplication map is not injective. Furthermore, we can prove that $M$ is indecomposable. In fact, suppose that $M = N_1 \oplus N_2$, then the indecomposable submodule generated by $y+I$ must be contained in one of these components, say $\langle y\rangle \subseteq N_1$ .

It is clear that $x^4+I$ is not in $N_1$, but neither in $N_2$, otherwise $x^4y+I \in N_1 \cap N_2$. Therefore, $x^4+I = (n_1+I) + (n_2 +I)\in N_1 \oplus N_2$.

Since $HF_{N_1}(3)=1$,  we get that $n_1+I = \alpha x^3y+I$, $\alpha \in k$, then $n_2+I= x^4 - \alpha x^3y +I$. This contradicts to the fact that $ yn_2+I =x^4y +I\in N_1$.
\end{exa}

%%%%%%%%%%%%%%%%%%%%%%%%%%%%%%%%%%%%%%%%%%
\section*{Acknowledgements}

The main work of this article was performed while the authors were participating
in the summer school PRAGMATIC 2011 in Catania, Sicily. The authors
would like to thank the organizers of PRAGMATIC 2011 for their support and hospitality. We are
grateful to Mats Boij and Ralf Fr\"oberg for their suggestion to study this
topic, and for sharing knowledge, insights and experience in many helpful discussions.

\end{document}